\documentclass[12pt]{ociamthesis}  

\usepackage[margin=1.1in]{geometry}
\usepackage{mathtools}
\usepackage{amsmath}
\usepackage{amssymb}
\usepackage{hyperref}
\usepackage{indentfirst}
\usepackage{tikz-cd}

\usepackage{amsthm}
\usepackage{cleveref}

\title{Geometric Class Field Theory}   

\author{Hanming Liu}             
\college{Balliol College}  

\degree{MMath in Mathematics}     
\degreedate{Trinity 2022}         

\everymath{\displaystyle}

\begin{document}

\baselineskip=18pt plus1pt

\setcounter{secnumdepth}{3}
\setcounter{tocdepth}{3}

\setlength\parindent{20pt}

\theoremstyle{plain}
\newtheorem{theorem1}{Theorem}

\theoremstyle{plain}
\newtheorem{theorem}{Theorem}[section] 
\newtheorem{lemma}[theorem]{Lemma}
\newtheorem{corollary}[theorem]{Corollary}
\newtheorem{proposition}[theorem]{Proposition}

\theoremstyle{definition}
\newtheorem{definition}[theorem]{Definition} 
\newtheorem{example}[theorem]{Example} 
\newtheorem{remark}[theorem]{Remark}

\crefname{lemma}{lemma}{lemmas}

\newcommand{\gj}{J_{\mathbf{m}}}
\newcommand{\md}{\mathbf{m}}
\newcommand{\sheaf}{\mathcal{O}}
\newcommand{\lsheaf}{\mathcal{L}}
\newcommand*{\defeq}{\mathrel{\vcenter{\baselineskip0.5ex \lineskiplimit0pt\hbox{\scriptsize.}\hbox{\scriptsize.}}}=}
\newcommand{\Tr}{\textrm{Tr}}
\newcommand{\ddeg}{\textrm{deg}}
\newcommand{\g}{\mathfrak{g}}
\newcommand{\spec}{\textrm{Spec}}
\newcommand{\m}{\mathfrak{m}}

\maketitle                  
\include{dedication}        
\include{acknowlegements}   
\include{abstract}          

\begin{romanpages}          
\tableofcontents            
\end{romanpages}            


\chapter{Introduction}

In this expository article we present Rosenlicht's work on geometric class field theory, which classifies abelian coverings of smooth, projective, geometrically connected curves over perfect fields. Unless otherwise specified, all proofs in this expository article are either taken directly from [1], or are adaptations of proofs in [1].

Most of the author's effort went into translating statements and proofs regarding generalised Jacobians from Weil's language of algebraic geometry into the language of schemes (in analogy to [4], where Milne translated statements and proof regarding Jacobian varieties from Weil's language into the language of schemes), whereas the rest of the story that don't use Weil's language in some essential way are mostly included here for the sake of completeness, and some proofs include added details for the sake of the author's understanding.

\section{Summary of the Story}

Let \(k\) be a perfect field, and let \(X\) be a smooth, projective, geometrically connected curve over \(k\). Let $\md$ be a modulus (which just means an effective divisor) on $X$. Write \(\mathbf{m}=\sum n_{P}P\) and $S$ for the support of $\md$ (that is, the set of points where $n_{P}\neq0$), and \(v_{P}\) for the valuation of a rational function at the point \(P\). For a rational function \(\phi\) on \(X\), we say \(\phi \equiv 1 \)  mod \(\mathbf{m}\) if \(v_{P}(1-\phi) \geq n_{P}\) for each \(P\).

Chapter 2 is devoted to the following theorem:

\begin{theorem1}
    Let \(f: X\dashrightarrow G\) be a rational map from \(X\) to an algebraic group \(G\) which is regular away from \(S\). This naturally extends to a map from divisors prime to \(S\) to \(G\). There exists a modulus \(\md\) with support \(S\), such that \(f(D)=0\) for every divisor \(D=(\phi)\) with \(\phi \equiv 1\) mod \(\mathbf{m}\).
\end{theorem1}

Chapter 3 is devoted to the following theorem:

\begin{theorem1}\label{mainthm2}
    There exists an algebraic group \(J_{\mathbf{m}}\) and a principal homogeneous space $\gj^{(1)}$, depending on \(X\) and \(\mathbf{m}\), and a rational map \(\varphi_{\md}:X\dashrightarrow \gj\) regular away from \(S\), satisfying the following universal property: for every rational map \(f:X\dashrightarrow G\) to a principal homogeneous space of a connected commutative algebraic group regular away from \(S\), there exists some modulus $\md$ with support $S$ such that there exists a unique morphism \(\theta:\gj^{(1)}\rightarrow H\) such that \(f=\theta\circ \varphi_{\md}\).
\end{theorem1}

A map \(f:H'\rightarrow H\) between principal homogeneous spaces of commutative algebraic groups is called an isogeny if it is surjective, and has finite kernel. A isogeny is called an abelian isogeny, if in addition to being an isogeny, it is an abelian covering.

\begin{theorem1}
    Every abelian covering of a variety is the pullback of an abelian isogeny.
\end{theorem1}

Combining the above 3 theorems, we get:

\begin{theorem1}
    Every abelian covering of \(X\) (irreducible, non-singular, projective curve) is the pullback of an abelian isogeny \(\theta :H\rightarrow\gj^{(1)}\).
\end{theorem1}

We will also prove that for every covering, there is a smallest $\md$ we can choose in Theorem 4, and in this case, the isogeny is unique. This classifies all abelian coverings of $X$.

\section{Conventions}

We will work with schemes only (unlike [1], which works with varieties over universal domains). Points will always mean closed points.

Projective curves over a field $k$ don't have to be embedded in some $\mathbf{P}_{k}^{n}$, they can be embedded in some projective space over a finite extension of $k$.

Varieties (and in particular, curves) are not assumed to be irreducible, although we will (explicitly) make this assumption most of the time.

\chapter{Maps From Curves to Commutative Groups}

$k$ denotes a perfect field. $X$ denotes a smooth, projective, geometrically connected curve. $\md$ denotes a modulus (effective divisor) on $X$, whose support is denoted $S$. $G$ denotes a connected commutative algebraic group.

Write \(\mathbf{m}=\sum n_{P}P\), and \(v_{P}\) for the valuation of a rational function at the point \(P\). For a rational function \(\phi\) on \(X\) and $a\in k$, we say \(\phi \equiv a \)  mod \(\mathbf{m}\) if \(v_{P}(a-\phi) \geq n_{P}\) for each \(P\) (if the inequality holds only for a point $P\in S$, then we say $\phi\equiv1$ mod $\md$ at $P$). Note that for any such $\phi$, $(\phi)$ is a divisor prime to $S$.

Let $f:X-S\rightarrow G$ be a (regular) map (which we will also treat as a rational map whose domain of definition is contained in $X-S$). $f$ extends to a map from the set of divisors on $X$ prime to $S$ to $G$.

\begin{definition}
    We say that $f$ admits the modulus $\md$ (or that $\md$ is a modulus of $f$) iff $f((g))=0$ for every function $g\in k(X)$ such that $g\equiv 1$ mod $\md$.
\end{definition}

The main goal of this chapter is to prove theorem 1, which we restate here:

\begin{theorem}\label{theorem1}
    Let $S$ be a finite set on a smooth, projective, geometrically connected curve $X$. Let $G$ be a connected commutative algebraic group. Every rational map $X\dashrightarrow G$ regular on $X-S$ admits a modulus $\md$ whose support is $S$.
\end{theorem}

We will only prove it in the case of positive characteristic. The proof of the case of characteristic zero is quite different, and can be found in [1], chapter 3, no 6.

\section{Local Symbols}

In this section, we define the notion of a local symbol, prove some elementary properties of local symbols, and give a criterion for admitting a modulus in terms of traces (\Cref{crit}).

For this section and the next, $k$ will be assumed to be algebraically closed.

\begin{definition}
    Let $\md$ be a modulus on $X$. Let $f:X-S\rightarrow G$. We call a ``local symbol'' the assignment, for each $P\in X$ and every $g\in k(X)^{*}$, of an element of $G$, written $(f,g)_{P}$, satisfying:
    \begin{enumerate}
        \item[(i)] $(f,gg')_{P}=(f,g)_{P}+(f,g')_{P}$.
        \item[(ii)] $(f,g)_{P}=0$ if $P\in S$ and if $g\equiv1$ mod $\md$ at $P$.
        \item[(iii)] $(f,g)_{P}=v_{P}(g)f(P)$ if $P\in X-S$.
        \item[(iv)] $\sum_{P\in X}(f,g)_{P}=0$.
    \end{enumerate}
\end{definition}

\begin{proposition}\label{prop2.1.2}
    $f$ admits a modulus $\md$ iff there exists a local symbol assiciated to $f$ and $\md$, and such symbol is then unique.
\end{proposition}
\begin{proof}
    (Sketch) If a local symbol exists (with respect to a modulus $\md$), and $g\equiv1$ mod $\md$, then $f((g))=0$ by properties (ii), (iii), (iv).
    
    If $f$ admits a modulus $\md$, then define a local symbol as follows: by (iii), define $(f,g)_{P}\defeq v_{P}(g)f(P)$ if $P\in X-S$. For $P\in S$, by approximation theorem of valuations, we can find an auxiliary function $g_{P}$ such that $g_{P}\equiv1$ mod $\md$ at all $Q\in S-P$, and that $g/g_{P}\equiv1$ mod $\md$, and define $(f,g)_{P}\defeq-\sum_{Q\notin S}v_{Q}(g_{P})f(Q)$.
    
    Verification that properties (i)-(iv) are satisfied can be found in [1].
\end{proof}

\begin{proposition}
    Let $f:X-S\rightarrow G$, and let $\theta:G\rightarrow G'$ be an algebraic homomorphism of commutative groups. If $f$ admits a modulus $\md$, then $\theta\circ f$ also admits $\md$ as a modulus, and the corresponding local symbols satisfy $(\theta\circ f,g)_{P}=\theta((f,g)_{P}).$ 
\end{proposition}

\begin{proof}
    We check that $(\theta\circ f,g)_{P}=\theta((f,g)_{P})$ is a local symbol.
    \begin{enumerate}
        \item[(i)] $\theta((f,gg')_{P})=\theta((f,g)_{P}+(f,g')_{P})=\theta((f,g)_{P})+\theta((f,g')_{P})$
        \item[(ii)] Let $g\equiv1$ mod $\md$ at $P$. $\theta((f,g)_{P})=\theta(0)=0$
        \item[(iii)] Let $P\in X-S$. $\theta((f,g)_{P})=\theta(v_{P}(g)f(P))=v_{P}(g)\theta(f(P))$
        \item[(iv)] $\sum_{P\in X}\theta((f,g)_{P})=\theta(\sum_{P\in X}(f,g)_{P})=\theta(0)=0$.
    \end{enumerate}
    
    So by \Cref{prop2.1.2}, $\md$ is a modulus of $\theta\circ f$.
\end{proof}

\begin{definition}\label{dtrace}
    Let $f:X-S\rightarrow G$, and let $\pi:X\rightarrow X'$ be a covering of curves. Put $S'=\pi^{-1}(S)$, and for every $P'\in X'$, denote by $\pi^{-1}(P')$ the divisor of $X$ which is the inverse image of $P'$ by $\pi$ counting multiplicity (so we have $\pi^{-1}(P')=\sum_{P\mapsto P'}e_{P}P$, where $e_{P}$ denotes the ramification index at $P$). If $P'\notin S'$, then $f(\pi^{-1}(P'))$ makes sense, and we get a map $$\Tr_{\pi}f:X'-S'\rightarrow G$$ which will be called the trace of $f$.
\end{definition}

\begin{proposition}\label{trace}
    With the setting of \Cref{dtrace}, if $f$ admits a modulus $\md$, then $f'\defeq \Tr_{\pi}f$ admits a modulus $\md'$, and for every $P'\in X'$ and $g\in k(X')^{*}$, $$(\Tr_{\pi}f,g)_{P'}=\sum_{P\mapsto P'}(f,g\circ\pi)_{P}.$$
\end{proposition}
\begin{proof}
    We check that $(\Tr_{\pi}f,g)_{P'}=\sum_{P\mapsto P'}(f,g\circ\pi)_{P}$ is a local symbol.
    \begin{enumerate}
        \item[(i)] $\sum_{P\mapsto P'}(f,gg'\circ\pi)_{P}=\sum_{P\mapsto P'}(f,g\circ\pi)_{P}+(f,g'\circ\pi)_{P}=\sum_{P\mapsto P'}(f,g\circ\pi)_{P}+\sum_{P\mapsto P'}(f,g'\circ\pi)_{P}$
        \item[(ii)] Write $\md=\sum_{P\in S}n_{P}P$. For every $P'\in S'$, choose an integer $n_{P'}$ such that $$n_{P'}>n_{P}/e_{P}$$ for all $P\in S\cap\pi_{-1}(P')$. Let $g\equiv1$ mod $\md'$ at $P$, where $$\md'\defeq\sum_{P'\in S'}n_{P'}P'.$$ We have that for every $P\mapsto P'$ and $P\in S$, $$v_{P}(1-g\circ\pi)\geq e_{P}n_{P'}\geq n_{P},$$ so $(f,g\circ\pi)_{P}=0$. If $P\mapsto P'$ but $P\notin S$, then since $v_{P}(g\circ\pi)=0$, we also have $(f,g\circ\pi)=0$. So, for $P'\in S'$, $$\sum_{P\mapsto P'}(f,g\circ\pi)_{P}=0.$$
        \item[(iii)] Let $P'\in X'-S'$. $$\sum_{P\mapsto P'}(f,g\circ\pi)_{P}=\sum_{P\mapsto P'}v_{P}(g\circ\pi)f(P)=v_{P'}(g)\sum_{P\mapsto P'}e_{P}f(P)=v_{P'}(g)\Tr_{\pi}f.$$
        \item[(iv)] $\sum_{P'\in X'}\sum_{P\mapsto P'}(f,g\circ\pi)_{P}=\sum_{P}(f,g\circ\pi)_{P}=0$
    \end{enumerate}
    So by \Cref{prop2.1.2}, $\md$ is a modulus for $\Tr_{\pi}f$.
\end{proof}

\begin{proposition}\label{tracereg}
    With the setting of \Cref{dtrace}, $\Tr_{\pi}f:X'-S'\rightarrow G$ is a regular map.
\end{proposition}
\begin{proof}
    Write $n$ for the degree of the covering $X\rightarrow X'$.
    
    Define $F:(X-S)^{n}\rightarrow G$ to be $$F(P_{1},...,P_{n})=f(P_{1})+...+f(P_{n}).$$ F is invariant under permutations of $(X-S)^{n}$, so $F$ factors through $(X-S)^{(n)}$, giving a regular map $F':(X-S)^{(n)}\rightarrow G$.
    
    For every $P'\in X'-S'$, we identify $\pi^{-1}(P')=\sum_{P\mapsto P'}e_{P}P$ with an element of $(X-S)^{(n)}$. The map $$\pi^{-1}:X'-S'\rightarrow (X-S)^{(n)}$$ is regular. 
    
    The composition $$F'\circ\pi^{-1}:X'-S'\rightarrow G$$ is then a regular map equal to $\Tr_{\pi}f$.
\end{proof}

A non-constant $g\in k(X)$ is a covering from $X$ to $\mathbf{P}_{k}^{1}$, so the above proposition applies. Write \(\mathbf{m}=\sum n_{P}P\). We write \(g \equiv 0 \)  mod \(\mathbf{m}\) if \(v_{P} \geq n_{P}\) for each \(P\). If $g\equiv 0$ mod $\md$, then $g(P)=0$ for all $P\in S$, and $S'=g(S)=\{0\}\subset \mathbf{P}_{k}^{1}$. Then $\Tr_{g}f$ is a regular map from $\mathbf{P}_{k}^{1}$ to $G$.

\begin{proposition}\label{crit}
    $f:X\dashrightarrow G$ admits a modulus $\md$ iff for every non-constant rational function $g$ with $g\equiv0$ mod $\md$, $\Tr_{g}f$ is constant.
\end{proposition}
\begin{proof}
    Assume $\md$ is a modulus of $f$. Let $g$ be a non-constant function with $g\equiv0$ mod $\md$. For every $a\in \mathbf{P}_{k}^{1}$, $\Tr_{g}f(a)=f(g^{-1}(a))$, which is $f((g)_{a})$. ($(g)_{a}$ is the divisor of zeroes of $g-a$ when $a\neq\infty$, and the divisor of poles of $g$ when $a=\infty$)
    
    If in addition $a\neq\infty$, then since $g\equiv0$ mod $\md$, we have that $g-a$ is $-a$ mod $\md$. By linearity of $f$ as a map from the set of divisors prime to $S$, and the fact that it admits $\md$ as a modulus, we have that $f((g-a))=0$, which means $\Tr_{g}f(a)=f((g)_{a})=f((g-a)_{\infty})=f((g)_{\infty})=\Tr_{g}f(\infty)$. So $\Tr_{g}f$ is constant.
    
    Conversely, assume $\Tr_{\pi}g$ is constant for every non-constant $g\equiv0$ mod $\md$. Let $h\equiv1$ mod $\md$. We can write $h=1-g$ with $g\equiv0$ mod $\md$. If $g$ is constant, then $h$ is also constant, and $f((h))=f(0)=0$. If $g$ is non-constant, then $(h)=(g)_{1}-(g)_{\infty}$ (the divisors of 0's of $h$ which are 1's of $g$, minus the divisor of poles of $g$). So $f((h))=\Tr_{g}f(1)-\Tr_{g}f(\infty)=0$.
\end{proof}

\section{Proof of Theorem 1 in Positive Characteristic}

Assume $\textrm{char}(k)>0$. We first quote two theorems about the structure of connected commutative algebraic groups.

\begin{theorem}\label{chevalley}
    (Chevalley) Every connected algebraic group $G$ contains a normal subgroup $R$ such that $R$ is a connected linear group, and  that $G/R$ is an Abelian variety.
\end{theorem}

\begin{theorem}\label{gmu}
    Every connected commutative linear algebraic group is isomorphic to the product of some copies of the multiplicative group $\mathbf{G}_{\md}$ and a unipotent group $U$ (that is, $U$ is isomorphic to a subgroup of the group of triangular matrices having only 1's on the principal diagonal).
\end{theorem}

We use these 2 theorems to reduce the proof as follows: let $f:X-S\rightarrow G$. Use \Cref{chevalley} to factor $G$ as $G/R=A$. Use \Cref{gmu} to write $R=R_{1}+R_{2}$, which gives that $G$ is embedded as a subgroup in $G/R_{1}\times G/R_{2}$.

If $\md_{i}$ (with support $S$) are moduli for the composed maps $X-S\rightarrow G\rightarrow G/R_{i}$ for each $i=1,2$, then $\md=\textrm{Sup}(\md_{1},\md_{2})$ is a modulus for $f$ with support $S$. It would thus suffice to prove Theorem 1 for $G/R_{1}$ and $G/R_{2}$, which are extensions of $A$ by $R_{1}$ and $R_{2}$. Repeating this process, and using \Cref{gmu}, it suffices to prove Theorem 1 for the following 2 cases:

\begin{enumerate}
    \item $G$ is an extension of an abelian variety $A$ by $\mathbf{G}_{m}$.
    \item $G$ is an extension of an abelian variety $A$ by a unipotent group $U$.
\end{enumerate}

First we prove case 1.

\begin{lemma}\label{casealemma2}
    Every regular map $\mathbf{P}_{k}^{1}-\{point\}\rightarrow \mathbf{G}_{m}$ is constant.
\end{lemma}
\begin{proof}
    We can assume that the point is $\infty$. In this case, we have a regular map $\mathbf{A}_{k}^{1}\rightarrow\mathbf{A}_{k}^{1}-\{0\}$, which is a polynomial with no zero, which is a constant.
\end{proof}

\begin{lemma}\label{casealemma}
    Every rational map from $\mathbf{P}_{k}^{1}$ to an abelian variety is constant.
\end{lemma}
\begin{proof}
    (This proof uses the theory of Jacobians, and is different from the proof given in [1]) By Theorem 3.2 of [3], any such rational map is regular (since $\mathbf{P}_{k}^{1}$ is non-singular). So such a map admits the modulus $0$.
    
    By the results of next chapter (which do not depend on results of this chapter), any such map factors through the Jacobian variety $J$ of $\mathbf{P}_{k}^{1}$, which is a point. So such a map is constant.
\end{proof}

\begin{proposition}
    Theorem 1 is true when $G$ is the extension of an abelian variety $A$ by $\mathbf{G}_{\md}$, taking $\md=\sum_{P\in S}P$.
\end{proposition}
\begin{proof}
    We use \Cref{crit}. Let $g$ be a non-constant rational function on $X$ such that $g\equiv0$ mod $\md$.
    
    $\textrm{Tr}_{g}f$ is a regular map from $\mathbf{P}_{k}^{1}-\{0\}$ to $G$ (by \Cref{tracereg}). Composing it with the projection $G\rightarrow A$, we get a rational map from $\mathbf{P}_{k}^{1}$ to $A$ which is constant by \Cref{casealemma}. So $\Tr_{g}f$ takes values in a coset of $\mathbf{G}_{m}$, which is isomorphic to $\mathbf{G}_{m}$ as varieties. By \Cref{casealemma2}, $\Tr_{g}f$ is constant. By \Cref{crit}, we are done.
\end{proof}

Next we prove case 2.

Write $p$ for the characteristic of $k$. Assume $G/U=A$, where $U$ is unipotent, and $A$ is an abelian variety. By considering an element $u\in U$ as $1+n$, n a nilpotent matrix in $\mathbf{GL}_{m}$. Then whenever $p^{r}\geq m$, we have $(1+n)^{p^{r}}=1+n^{p^{r}}=1$. On the other hand, according to a theorem of Weil, $p^{r}A=A$ (here we write the group law additively).

Put $U'=G/p^{r}G,G'=A\times U'$, and let $\theta:G\rightarrow G'$ be the product of the maps $G\rightarrow A$ and $G\rightarrow U'$. $p^{r}U'=0$, so it has no subgroup isomorphic to an abelian variety, or $\mathbf{G}_{m}$, and so by \Cref{gmu} it is nilpotent.

$\textrm{Ker}(\theta)=U\cap p^{r}G$, and $p^{r}G$ is an abelian variety (since $p^{r}U=0$, $p^{r}:G\rightarrow G$ factors as $G\rightarrow A\rightarrow G$). So $\textrm{Ker}(\theta)$ being the intersection of an affine variety with a complete variety, is a finite set.

\begin{lemma}
    If $\theta:G\rightarrow G'$ is an algebraic homomorphism of commutative groups with finite kernel, then if $\md$ is a modulus for $\theta\circ f$, it is also a modulus for $f$.
\end{lemma}
\begin{proof}
    We use \Cref{crit}. If $g$ is a non-constant rational function on $X$ such that $g\equiv0$ mo $\md$, then the map $\Tr_{g}(\theta\circ f):\mathbf{P}_{k}^{1}-\{0\}\rightarrow G'$ is constant (by \Cref{crit}). We also have that $\Tr_{g}(\theta\circ f)=\theta\circ \Tr_{g}f$, which shows that $\Tr_{g}f$ takes values in a coset of $\textrm{Ker}(\theta)$, which is finite. As $\mathbf{P}_{k}^{1}-\{0\}$ is connected, it follows that $\Tr_{g}f$ is constant, so by \Cref{crit}, $\md$ is also a modulus for $f$.
\end{proof}

So, by the same reduction as in the beginning of this section, it suffices to prove Theorem 1 when $G$ is an abelian variety, and when $G$ is a unipotent group.

If $G$ is an abelian variety, then combining \Cref{casealemma} and \Cref{crit}, we get Theorem 1.

Assume $G$ is a connected, unipotent, commutative group. We consider it as embedded in a linear group $\mathbf{GL}_{r}$, such that each of its elements is of the form $1+N$, where $N=(n_{ij})$ is a strictly upper triangular matrix. So we can consider $G$ as embedded in $\mathbf{M}_{r\times r}$. Every rational map $g$ from a $Y$ to $G$ is then equivalent to $r^{2}$ rational functions $g_{ij}$. For each $Q\in Y$, put $$w_{Q}(g)=\textrm{Sup}(0,-v_{P}(g_{ij}))$$ ($w_{Q}$ ``counts the order of a pole''). In particular this applies to $f:X-S\rightarrow G$. Choose an integer $n>0$ such that $n>(r-1)w_{Q}(f)$ for all $Q\in S$.

\begin{lemma}\label{finallemma}
    If $f^{\alpha}$ is a family of rational maps from a curve $Y$ to $G$, then (composition law written additively) $$w_{Q}(\sum f^{\alpha})\leq(r-1)\textrm{Sup}_{\alpha}w_{Q}(f^{\alpha}).$$
\end{lemma}
\begin{proof}
    Write the composition law on $G$ multiplicatively. Write $$\prod f^{\alpha}=\prod(1+n^{\alpha})$$ where $n^{\alpha}$ are a family of rational maps of $Y$ into the space of strictly upper triangular matrices. Since the product of any $r$ such marices is $0$, we can write the right hand side of the above equation in the form $$\sum\limits_{0\leq k\leq r-1}\sum\limits_{\alpha_{1}<...<\alpha_{k}}n^{\alpha_{1}}...n^{\alpha_{k}}.$$ So the components $(\prod f^{\alpha})_{ij}$ are polynomials in the $n^{\alpha}$'s of total degree $\leq r-1$. So by the definition of $w_{Q}$, this gives $$w_{Q}(\prod f^{\alpha})\leq(r-1)\textrm{Sup}_{\alpha}(w_{Q}(n^{\alpha}))=(r-1)\textrm{Sup}_{\alpha}(w_{Q}(f^{\alpha}))$$ as was to be shown (the left hand side looks differently from the statement of the lemma because here we are writing the composition law multiplicatively).
\end{proof}

\begin{proposition}
    $\md=\sum_{P\in S}nP$ is a modulus for $f$ ($n$ is the integer defined right before \Cref{finallemma}). Thus Theorem 1 is true when $G$ is the a unipotent group.
\end{proposition}
\begin{proof}
    We use \Cref{crit}. Let $g$ be a non-constant rational function on $X$ such that $g\equiv0$ mod $\md$. Write $$f'\defeq \Tr_{g}f:\mathbf{P}_{k}^{1}\dashrightarrow G.$$ We will show that $w_{0}(f')<1$, which shows that all $v_{0}(f'_{ij})\geq0$, which means that $f'$ is regular at $0$. By \Cref{tracereg}, it is regular everywhere. So it is constant (this can be seen by factoring through the Jacobian $J$ of $\mathbf{P}_{k}^{1}$, which is a point).
    
    Let $Y$ be a normal covering of $\mathbf{P}_{k}^{1}$ dominating $X$, and let $\pi:Y\rightarrow X$ be the projection. For any point $P$ in any covering of $\mathbf{P}_{k}^{1}$, denote by $e_{P}$ its ramification index. Choose a point $Q\in Y$ in the preimage $\pi^{-1}(S)$ such that $e\defeq e_{Q}$ is maximal. Put $P=\pi(Q)\in S$, then since $e_{P}=v_{P}(g)\geq n$ (which is the definition of $g\equiv0$ mod $\md$), $$w_{Q}(f'\circ g\circ \pi)=ee_{P}w_{0}(f')\geq new_{0}(f').$$
    
    Let $p^{k}$ be the inseparable degree of the extension $k(X)|k(\mathbf{P}_{k}^{1})$, we have that
    
        $$f'\circ g\circ\pi=p^{k}\sum f\circ\pi\circ\sigma_{j}$$ where $\sigma_{j}$ are certain elements of the Galois group of $Y$.
    
    Using \Cref{finallemma}, and putting $Q_{j}=\sigma_{j}Q$, $P_{j}=\pi(Q_{j})$, we get $$w_{Q}(f'\circ g\circ\pi)\leq(r-1)\textrm{Sup}_{j}w_{Q}(f\circ\pi\circ\sigma_{j})$$ $$\leq(r-1)\textrm{Sup}_{j}w_{Q}(f\circ\pi)\leq(r-1)\textrm{Sup}_{j}(e_{Q_{j}}w_{P_{j}}(f)).$$
    
    If $P_{j}\notin S$, we have $w_{P_{j}}(f)=0$ (since in this case, $P_{j}$ is not a ``pole''). So we can only consider the terms in $(r-1)\textrm{Sup}_{j}(e_{Q_{j}}w_{P_{j}}(f))$ with $P_{j}\in S$. For these terms, we have $e_{Q_{j}}\leq e$ (by the maximality of $e$ among these terms), and $w_{P_{j}}(f)<n/(r-1)$ (by the definition of $n$). So $$w_{Q}(f'\circ g\circ\pi)\leq(r-1)\textrm{Sup}_{j}(e_{Q_{j}}w_{P_{j}}(f))\leq en.$$
    
    So, we have $w_{0}(f')<1$, which is what we wanted.
\end{proof}

\section{Over Perfect Fields}

Let $k$ be any perfect field. Let $f:X\dashrightarrow G$ be a rational map of $k$-varieties regular on $X-S$ (where $X$ is a smooth, projective, geometrically connected curve, and $G$ is a connected commutative algebraic group, both over $k$).

Previously we have shown that $f_{\bar{k}}:X_{\bar{k}}\dashrightarrow G_{\bar{k}}$ admits a modulus $\md'$ with support $S_{\bar{k}}$. Define $\bar{\md}$ to be the sum of all the conjugates of $\md$ in the galois group $\textrm{Aut}(\bar{k}|k)$. Since $\bar{\md}$ is invariant under $\textrm{Gal}(\bar{k}|k)$ and $k$ is perfect, $\bar{\md}$ corresponds to a modulus $\md$ on $X$ with support $S$.

Since $\bar{\md}\geq\md$, it is a modulus for $f_{\bar{k}}$. Let $g\in k(X)$ be a rational function such that $g\equiv1$ mod $\md$. Then, $g_{\bar{k}}$ is a rational function on $X_{\bar{k}}$ with $g_{\bar{k}}\equiv1$ mod $\bar{\md}$, which gives that $f_{\bar{k}}((g_{\bar{k}}))=0$, which gives $f((g))=0$, which shows that $\md$ is indeed a modulus for $f$. Thus Theorem 1 is true in the case of perfect base fields.

\chapter{Generalised Jacobians}

In this chapter we construct the generalised Jacobians \(\gj\) and prove their universal property. We follow the original construction by Rosenlicht given in [5] (and Serre's exposition given in [1]), which is analogous to Weil's construction of the Jacobian variety ([4], section 7), but translated into the language of schemes.

Fix a smooth, projective, geometrically connected curve $X$ over a perfect field $k$, a modulus $\md=\sum n_{P}P$ on $X$, and denote by $S$ the support of $\md$.

\section{Preliminaries}

\begin{definition}
    If $\mathrm{deg}(\md)=\sum n_{P}[k(P):k]\geq2$, define the curve $X_{\md}$ by the following: the set of $X_{\md}$ is $(X-S)\cup Q$ (crunching all points of $S$ into one point, $Q$), and the sheaf is defined by specifying the local rings: for each point $P$ in $X-S$, the local ring $\sheaf_{X_{\md},P}$ is the same as $\sheaf_{X,P}$. For $Q$, the local ring $\sheaf_{X_{\md},Q}$ is formed by the $k$-algebra generated by functions $f$ on $X$ such that $f\equiv0$ mod $\md$ (that is, $v_{P}(f)\geq n_{P}$ for each $P\in S$).
    
    If $\mathrm{deg}(\md)=0,1$, then define $X_{\md}=X$.
\end{definition}

\begin{remark}
    The arithmetic genus of $X_{\md}$ is equal to $g+\mathrm{deg}(\md)-1$ (if $\md\neq0$). We will denote $\pi$ for this quantity.
\end{remark}

\begin{definition}
    Any divisor $D$ on $X$ prime to $S$ corresponds to a line bundle $\lsheaf(D)$. Since $X-S$ is isomorphic to $X_{\md}-\{Q\}$, we can transport this sheaf to $X_{\md}-\{Q\}$. Define a sheaf $\lsheaf_{\md}$ on $X_{\md}$ by its local rings as follows: $\lsheaf_{\md}(D)_{Q}=\sheaf_{X_{\md},Q}$, and for $P\neq Q$, $\lsheaf_{\md}(D)_{P}=\lsheaf(D)_{Q}$. $\lsheaf_{\md}(D)$ will also be used to denote the pullback sheaf on $X$.
    
    Similarly to the nonsingular case, we put $$L_{\md}(D)=H^{0}(X_{\md},\lsheaf_{\md}(D)),l_{\md}(D)=\mathrm{dim}(L_{\md}(D)),$$ $$I_{\md}(D)=H^{1}(X_{\md},\lsheaf_{\md}(D)),i_{\md}(D)=\mathrm{dim}(I_{\md}(D)).$$
\end{definition}

The nonsingular and singular Riemann-Roch theorems are crucial throughout this chapter:

\begin{theorem}
    (Riemann-Roch) For every divisor $D$ on $X$, \[l(D)-i(D)=\mathrm{deg}(D)+1-g.\]
\end{theorem}

\begin{theorem}
    (Singular Riemann-Roch) For every divisor $D$ on $X$ prime to $S$, \[l_{\md}(D)-i_{\md}(D)=\mathrm{deg}(D)+1-\pi.\]
\end{theorem}

\section{Construction}

Since we are dealing with a curve with a modulus (or equivalently, a curve with 1 singularity), we need a refinement of the notion of equivalence of divisors:

\begin{definition}\label{defmeq}
    Let \(D\) and \(D'\) be divisors prime to \(S\). \(D\) and \(D'\) are said to be \(\md\)-equivalent (written \(D\equiv_{\md}D'\)), iff there exists a rational function \(g\) on \(X\) satisfying:
    
    \begin{enumerate}
    \item[(a)] \(g\equiv 1\) mod \(\md\) (suppressed iff \(\md=0\))
    
    \item[(b)] \((g)=D-D'\)
    \end{enumerate}

\end{definition}

\begin{lemma}\label{lemmameq}
    Suppose \(\md\neq 0\). Let \(D\) be a divisor prime to \(S\). The set of effective divisors \(D'\) \(\md\)-equivalent to \(D\) is bijective to \(\mathbf{P}(L_{\md}(D))\setminus \mathbf{P}(L(D-\md))\).
\end{lemma}
\begin{proof}
    An effective divisor \(D'\) \(\md\)-equivalent to \(D\) corresponds to a rational function \(g\) satisfying the conditions in \Cref{defmeq}. Such a \(g\in \sheaf'_{Q}\) in \(X_{\md}\), so \(g\in L_{\md}(D)\). In addition to \(g\in L_{\md}(D)\), we also need \(g\) to be nowhere zero on \(S\), which in \(L_{\md}(D)\), means that \(g\) is not in \(L(D-\md)\).
\end{proof}

We now begin to construct a rational composition law on \(X^{\pi}\) using the singular Riemann-Roch theorem.

Let $\md$ be a divisor on $X$, and assume that $k$ is algebraically closed, thus there exists a point $P$ outside the support $S$ of $\md$ (we will make this assumption for the first 3 sections of this chapter, and deal with the general case in section 4). Fix this point, which will act as our ``base point''. 

\begin{lemma}\label{minus1}
    
        Let \(D\) be a divisor on \(X\) prime to \(S\) such that \(i_{\md}(D)>0\), then there is a nonempty open set \(U\) of \(X-S\) such that \(i_{\md}(D+Q)=i_{\md}(D)-1\) for all \(Q\) in \(U\), and \(i_{\md}(D+Q)=i_{\md}(D)\) for \(Q\) not in \(U\).

\end{lemma}
\begin{proof}
    We follow the proof of Lemma 5.2 in [4]. If \(Q\) is not in the support of \(D\), then by the singular Serre duality theorem (see the definitions of differentials on a singular curve and the statement of this theorem in [1], chapter 4, no 10), \[I_{\md}(D+Q)^{*}\cong \Omega_{\md}(D+Q).\] So, $I_{\md}(D+Q)^{*}$ can be identified with the subspace of \(\Omega_{\md}(D)\) of differentials with a zero at \(Q\). We thus take \(U\) to be the complement of the zero set of \(I_{\md}(D)\), minus a subset of the support of D.
\end{proof}

\begin{lemma}\label{lemmameq2}
    For every \(r\leq\pi\), there is an open set \(U\) of \((X-S)^{r}\) such that \(l_{\md}(\sum P_{i})=1\) for all \((P_{1},...,P_{r})\) in \(U\).
\end{lemma}
\begin{proof}
    By the singular Riemann-Roch theorem, \(i_{\md}(0)=\pi\). By applying \Cref{minus1} \(r\) times, we obtain an open set \(U\) of \((X-S)^{r}\) such that \[i_{\md}(\sum P_{i})=\pi-r\] for all \((P_{1},...,P_{r})\) in \(U\). Again by the singular Riemann-Roch theorem, \[l_{\md}(\sum P_{i})=r+(1-\pi)+(\pi-r)=1\] for all \((P_{1},...,P_{r})\) in \(U\).
\end{proof}

\begin{lemma}\label{compositionlemma}
    Let $r$ be a positive integer.
    
    (a) Assume $r\geq\pi$. Let $D'$ a divisor of degree $\pi-2r$ on $X$ prime to $S$ such that $-D'$ is effective. There exists an open set \(U\) in \((X-S)^{(r)}\times (X-S)^{(r)}\) such that for all field extensions \(K|k\) and all \((N,M)\) in \(U(K)\), there exists a unique effective divisor \(\md\)-equivalent to \(N+M+D'\).
    
    (b) Let $D'$ be an effective divisor of degree $\pi$ on $X$ prime to $S$. There exists an open set \(V\) in \((X-S)^{(r)}\times (X-S)^{(r)}\) such that for all field extensions \(K|k\) and all \((N,M)\) in \(V(K)\), there exists a unique effective divisor \(\md\)-equivalent to \(N-M+D'\).
\end{lemma}
\begin{proof}
    We prove (a) by finding an open set \(U'\) in \((X-S)^{(2r)}\) such that for all field extensions \(K|k\) and all \((D)\) in \(U(K)\), there exists a unique effective divisor \(\md\)-equivalent to \(D+D'\), and then take $U$ to be the preimage of \(U'\) under the quotient map \((X-S)^{(r)}\times (X-S)^{(r)}\rightarrow (X-S)^{(2r)}\). The proof of (b) is analogous.
    
    By \Cref{lemmameq}, we need to find an open set \(U'\) in \((X-S)^{(2r)}\) such that for all field extensions \(K|k\) and all \((D)\) in \(U(K)\):
    \begin{enumerate}
        \item[(1)]\(l_{\md}(D+D')=1\)
        \item[(2)]\(l(D+D'-\md)=0\) (suppressed iff \(\md= 0\)).
    \end{enumerate}
    We will find an open set for each condition, and then take their intersection.
    
    For (1), we adapt the proof of Lemma 7.1 in [4]. Let \(D_{can}\) be the canonical relative effective divisor on \[(X-S)\times (X-S)^{(2r)}/(X-S)^{(2r)}.\] By the singular Riemann-Roch theorem, \(L_{\md}(D+D')\geq 1\) for all \(D\in (X-S)^{(2r)}\). So the semicontinuity theorem shows that the subset \(V\) of \((X-S)^{(2r)}\) of points \(t\) such that \(l_{\md}((D_{can})_{t}+D')=1\) is open.
    
    To show that \(V\) is nonempty, we obtain from \Cref{lemmameq2} effective divisors \(D\) of degree \(\pi\) with \[l_{\md}((D-D')+D')=l_{\md}(D)=1.\] 
    
    For (2), assume that \(\md\neq 0\). Suppose for the sake of contradiction that \[l(D+D'-\md)>0.\] The degree of \(D+D'-\md\) is \(g-1\), and the (nonsingular) Riemann-Roch theorem shows that \[i(D+D'-\md)=l(D+D'-\md)>0.\] Applying Lemma 3.1.3 (a) (with \(\md=0\)) \(\pi\) times, we obtain an open set \(V'\) of \((X-S)^{(2r)}\) such that for all $(D)\in V'(K)$, \[i(D+D'-\md)=i(D''+D'-\md)+\pi,\] where \(D-D''\) is an effective divisor of degree \(\pi\). So \[i(D''+D'-\md)>\pi.\] But \(D''+D'-\md\) has negative degree, so \[l(D''+D'-\md)=0,\] and the Riemann-Roch theorem gives \[i(D''+D'-\md)=\pi,\] a contradiction.
    
    Now take \(U'=V\cap V'\), and \(U\) to be the preimage of \(U'\) under the quotient map \((X-S)^{(r)}\times (X-S)^{(r)}\rightarrow (X-S)^{(2r)}\).
\end{proof}

\begin{remark}\label{compositionremark}
    By the semicontinuity theorem, the set of $D$ satisfying conditions (1) and (2) in the above proof forms an open set. That is, not only do we get an open set that does the job, we further have that all points $(N,M)$ such that there exists a unique effective divisor $\md$-equivalent to $N+M+D'$ (resp. $N-M+D'$) form an open set.
\end{remark}

\begin{theorem}\label{thmrgl}
    There exists a unique rational composition law \[m:X^{(\pi)}\times X^{(\pi)} \dashrightarrow X^{(\pi)}\] whose domain of definition contains any open set \(U\) given by \Cref{compositionlemma} (a), and for all field extensions \(K|k\) and all \(D,D')\) in \(U(K)\), \[m(D,D')\sim_{\md} D+D'-\pi P.\]
    
    Moreover, \(m\) is a birational group law, that is, it is associative, and the maps \[(D,D')\mapsto (D,m(D,D')),(D,D')\mapsto (m(D,D'),D')\] are birational.
\end{theorem}
\begin{proof}
    We adapt the proof of Proposition 7.2 in [4]. With \Cref{compositionlemma} (a) (taking $r=\pi$), it is clear how to define \(m\) and prove its uniqueness. In order to prove the existence of \(m\), we just need to check that it is a map of varieties (schemes). To do this, we use the Yoneda lemma.
    
    Let \(T\) be an integral \(k\)-scheme. An element of \(U(T)\) is a pair of relative effective divisors \((D,D')\) on \((X-S)\times T/T\) such that, for all \(t\in T\), \(l_{\md}(D_{t}+D'_{t}-\pi P)=1\), and \(l(D_{t}+D'_{t}-\pi P-\md)=0\) (suppressed iff \(\md=0\)). Write \[\lsheaf \defeq\lsheaf_{\md}(D+D'-\pi P\times T),\] and \(q:(X-S)\times T\rightarrow T\) for the projection map. ([3], 4.2d) shows that \(q_{*}(\lsheaf)\) is a line bundle on \(T\). The canonical map \(q^{*}q_{*}\lsheaf \rightarrow \lsheaf\) tensored with \((q^{*}q_{*}\lsheaf)^{-1}\) gives a canonical global section \[s:\sheaf_{(X-S)\times T}\rightarrow \lsheaf \otimes (q^{*}q_{*}\lsheaf)^{-1},\] which determines a relatve effective divisor \(m(D,D')\) of degree \(\pi\) on \((X-S)\times T/T\). This construction is functorial, so by the Yoneda lemma, it is represented by a map of schemes. By considering the case where \(T\) is the spectrum of a field, we get \(m(D,D')\sim_{\md} D+D'-\pi P\).
    
    It remains to prove that \(m\) is a birational group law. Associativity follows from the observation that \[m(D,m(D',D'')),m(m(D,D'),D'')\sim_{\md} D+D'+D''-2\pi P.\]
    
    A similar construction using \Cref{compositionlemma} (b) gives a map \(r:X^{(\pi)}\dashrightarrow X^{(\pi)}\) such that \((p,r)\) (\(p\) denotes the projection of \(X^{(\pi)}\times X^{(\pi)}\) onto its first factor) is a birational inverse to \[(D,D')\mapsto (D,m(D,D')).\] By the commutativity of the composition law (inherited from the commutativity of addition of divisors), \[(D,D')\mapsto (m(D,D'),D')\] also has a birational inverse. 
\end{proof}

Now we cite a theorem of Weil ([6]), which allows us to pass from a birational group law to an algebraic group.

\begin{theorem}
    For any variety \(V\) over \(k\) with a birational group law, there exists a unique (up to unique isomorphism) algebraic group \(G\) over \(k\) and a birational map \(f:V\dashrightarrow G\) satisfying \(f(ab)=f(a)f(b)\) wherever \(ab\) is defined.
\end{theorem}

\begin{definition}
    (Generalised Jacobian) For \(X\) and \(\md\), the generalised Jacobian \(\gj\) is defined to be the unique algebraic group corresponding to the birational group law \(m\) in \Cref{thmrgl}. The birational map \(X^{(\pi)}\dashrightarrow \gj\) is denoted \(\varphi\).
\end{definition}

\begin{remark}
    Although the Jacobian variety of a curve (\(\gj\) where \(\md=0\)) is projective, and thus an abelian variety, the generalised Jacobian (the case of a nonzero modulus) is only quasi-projective.
\end{remark}

\section{Canonical Map}

Our first goal is to define a homomorphism from the group of divisors on \(X\) prime to \(S\) to \(\gj\). The following analogue of \Cref{compositionlemma} will be useful.

\begin{lemma}\label{precompositionlemma}
    Let \(D\) be a divisor on \(X\) of degree 0 and prime to \(S\). There exists an open set \(U\) in \((X-S)^{(\pi)}\) such that for all field extensions \(K|k\) and all \((D')\) in \(U(K)\), there exists a unique effective divisor \(\md\)-equivalent to \(D+D'\).
    
\end{lemma}
\begin{proof}
    By \Cref{lemmameq}, we need to find an open set \(U\) in \((X-S)^{(\pi)}\) such that for all field extensions \(K|k\) and all \((D')\) in \(U(K)\):
    \begin{enumerate}
        \item[(1)]\(l_{\md}(D+D')=1\)
        \item[(2)]\(l(D+D'-\md)=0\) (suppressed iff \(\md= 0\)).
    \end{enumerate}
    We will find an open set for each condition, and then take their intersection.
    
    For (1), \(deg(D)=0\) gives \(l_{\md}(D)\leq 1\). So the singular Riemann-Roch theorem gives \(0<i_{\md}(D)\leq\pi\). Apply \Cref{minus1} \(\pi\) times, we get an open set \(V\) such that \(i_{\md}(D+\sum Q_{i})=0\) for all \(Q_{1},...,Q_{\pi}\in V(K)\). The singular Riemann-Roch theorem then shows that on \(V(K)\), \(l_{\md}(D+\sum Q_{i})=1\).
    
    For (2), the proof is identical to that of \Cref{compositionlemma}. Assume that \(\md\neq 0\). Suppose for the sake of contradiction that \(l(D+D'-\md)>0\). The degree of \(l(D+D'-\md)\) is \(g-1\), and the (nonsingular) Riemann-Roch theorem shows that \(i(D+D'-\md)=l(D+D'-\md)>0\). Applying \Cref{minus1} (with \(\md=0\)) \(\pi\) times, we obtain an open set \(V'\) of \((X-S)^{(2\pi)}\) such that for all $(D)\in V'(K)$ \(i(D+D'-\md)=i(D+D''-\md)-\pi\), for any \((D',D'')\in V'(K)\) where \(D'-D''\) is an effective divisor of degree \(\pi\). So \(i(D+D''-\md)>\pi\). But \(D+D''-\md\) has negative degree, so \(l(D+D''-\md)=0\), and the Riemann-Roch theorem gives \(i(D+D''-\md)=\pi\), a contradiction.
    
    Taking \(U=V\cap V'\), we get the required open set.
\end{proof}

Now let \(D\) be any divisor on \(X\) prime to \(S\). \(D-\textrm{deg}(D)P\) is then a divisor of degree \(0\) prime to \(S\). Applying \Cref{precompositionlemma}, we get an open set \(U\) in \((X-S)^{(\pi)}\) such that for all field extensions \(K|k\) and all \((M)\) in \(U(K)\), there is a unique effective divisor \(N\) such that \[N\sim_{\md}D-\textrm{deg}(D)P+M.\] For each \((M)\) in \(U(K)\), define \[\theta_{M} (D)=\varphi(N)-\varphi(M).\]

With what we know at this point, $\varphi(N)$ might not always be defined, but for now let's assume that it is defined (and it is in fact always defined). This issue will be dealt with after \Cref{continuous}.

\begin{lemma}\label{thetahom}
    For any divisors \(D=D'+D''\) on \(X\) prime to \(S\), and any \(M,M',M''\) such that \(\theta_{M}(D)\), \(\theta_{M'}(D')\), \(\theta_{M''}(D'')\) make sense, \(\theta_{M}(D)=\theta_{M'}(D')+\theta_{M''}(D'')\). Furthermore, \(\theta_{M}(D)\) does not depend on \(M\).
\end{lemma}
\begin{proof}
    Let \(N,N',N''\) be the corresponding effective divisors for \(D,D',D''\) and
    
    \(M,M',M''\). We need to show that \[\varphi(N)-\varphi(M)=\varphi(N')-\varphi(M')+\varphi(N'')-\varphi(M''),\] which is equivalent (this can be seen from rearranging the terms so that each side is an effective divisor, and then rearrange back) to \[N-M\sim_{\md}N'-M'+N''-M'',\] which follows from \[D-\textrm{deg}(D)P\sim_{\md}D'-\textrm{deg}(D')P+D''-\textrm{deg}(D'')P.\]
    
    Taking \(D=D'=D''=0\), we get \(\theta_{M}(0)=0\).
    
    Taking \(D=D',D''=0\), we get \(\theta_{M}(D)=\theta_{M'}(D)\).
\end{proof}

We will write \(\theta\) for \(\theta_{M}\), since there is no \(M\) dependence. \Cref{thetahom} shows that \(\theta\) is a homomorphism. From the construction of \(\theta\) it is also clear that for any \((D)\in (X-S)^{(\pi)}\) for which \(\varphi(D)\) is defined, \(\theta(D)=\varphi(D)\).

\begin{proposition}\label{gjdvs}
    $\theta$ as a map from the group of divisors on $X$ prime to $S$ is surjective, and its kernel consists exactly of the divisors $\md$-equivalent to an integer multiple of $P$.
\end{proposition}
\begin{proof}
     The image of $\theta$ certainly contains the image of $\varphi$, so it is open (and thus dense) in $\gj$. So it is the entirety of $\gj$.
    
    For the kernel, $\theta(D)=0\Leftrightarrow \varphi(N)=\varphi(M)\Leftrightarrow M\sim_{\md}N\Leftrightarrow D\sim_{\md}\textrm{deg}(D)P$.
\end{proof}

\begin{lemma}\label{continuous}
    Let $D=P_{0}$ for some $P_{0}\in X-S$. Then in the above process, $M\mapsto N$ is a rational map on $X^{(\pi)}$.
\end{lemma}
\begin{proof}
    (This proof is not from [1], as there is no analogous statement of this lemma in [1]) The map $M\mapsto N$ is the map $M\mapsto m(M,P_{0}+(\pi-1)P)$, where $m$ is the rational composition law given by \Cref{thmrgl}.
    
    (There exists some $M$ such that $(M,P_{0}+(\pi-1)P)$ is in the domain of definition of $m$, because there exists some $M$ such that the corresponding unique $N$ exists, and then we can apply \Cref{compositionremark})
\end{proof}

Write $V$ for domain of definition of $\varphi$ in $X^{(\pi)}$. In the case where $D=P_{0}$ for some $P_{0}\in X-S$, we can intersect $U$ (this is the $U$ used in the procedure for defining $\theta_{M}(D)$, right before \Cref{thetahom}) with the preimage of $V$ under the map $M\mapsto N$, and we find that $\varphi(N)$ would indeed be defined.

Now with a general $D$, \Cref{thetahom} tells us that by taking linear combinations, $\theta(D)$ (and the associated $\phi(N)$) is always defined.

Next we cite an auxiliary lemma, whose proof can be found in [1], chapter 5, section 2, no 8.

\begin{lemma}\label{g1covering}
    Let $X,X'$ be projective nonsingular curves, and let $g:X\rightarrow X'$ be a separable covering of degree $n+1$. For every $P\in X$, the divisor $g^{-1}(g(P))$ is of the form $P+H_{P}$, where $H_{p}$ is an effective divisor of degree $n$. If $H_{P}$ is identified with a point of $X^{(n)}$, the map $P\mapsto H_{P}$ is a regular map from $X\rightarrow X^{(n)}$.
\end{lemma}

\begin{proposition}\label{thetaregular}
    \(\theta:X-S\rightarrow \gj\) is a regular map.
\end{proposition}
\begin{proof}
    The case $\pi=0$ is trivial, as $X^{(0)}$ (and so $\gj$) consists of a single point.
    
    Assume \(\pi\neq0\). Let \(P_{1}\in X-S\) be a point distinct from \(P\). By \Cref{precompositionlemma}, there exists an open set of \(M\in (X-S)^{(\pi)}\) such that there is a unique \(N\in (X-S)^{(\pi)}\) that is \(\md\)-equivalent to \(-P_{1}+P+M\). Fix such an \(M\) for now. By \Cref{compositionremark}, there is an open set $U$ containing $P_{1}$ such that all points $P_{i}$ in it, there is a unique $N_{i}$ $\md$-equivalent to $-P_{i}+P+M$.
    
    Let \(g\) be a rational function that witnesses such \(\md\)-equivalence, that is, \[(g)=N+P_{1}-P-M,g\equiv 1\textrm{ mod }\md.\] By the proof of \Cref{precompositionlemma}, and the fact that \(g\equiv 1\) mod \(\md\), such \(g\) is unique.
    
    Since \(N,P_{1},P,M\) are all effective divisors, we know that the divisor \((g)_{\infty}\) of poles of \(g\) is less than or equal to \(P+M\). We will show that \((g)_{\infty}=P+M\).
    
    We can assume that \(N-P\) is not effective, since we can intersect \(U\) with the complement of the preimage of $\pi P$ under the continuous (\Cref{continuous}) map \(M\mapsto N\). \(P_{1}-P\) is clearly also not effective. So, \(P\) is a pole of \(g\).
    
    We can in addition assume that \(M=\sum_{1}^{\pi}M_{i}\) for distinct \(M_{i}\) by further intersecting \(U\) by a suitable open set. We need to show that any \(M_{i}\) is a pole of \(g\). We will show that $M_{\pi}$ is a pole, and by analogy the other ones are poles as well.
    
    Assume that \(M_{\pi}\) is not a pole. We have $$g\in L_{\md}(P-P_{1}+\sum_{i<\pi}M_{i}),l_{\md}(P-P_{1}+\sum_{i<\pi}M_{i})\geq 1.$$ So by the singular Riemann-Roch theorem, $$i_{\md}(P-P_{1}+\sum_{i<\pi}M_{i})\geq1.$$
    
    By intersecting $U$ with the open set from \Cref{minus1} (taking $D=P-P_{1}$), we get that $$i_{\md}(P-P_{1})\geq\pi.$$ Again by the singular Riemann-Roch theorem, we get $$l_{\md}(P-P_{1})\geq1,$$ which gives a rational function $h$ such that $(h)=P-P_{1}$ and $h\equiv1$ mod $\md$. $(h)=P-P_{1}$ means that $h$ is an isomorphism between $X$ and $\mathbf{P}^{1}_{k}$, so the genus $g=1$. $h\equiv1$ mod $\md$ is only possible if $\textrm{deg}(\md)=1$, or if $\md=0$ (both of which imply $\pi=0$, which contradicts our assumption at the beginning of the proof). So, we have shown that $M_{\pi}$ is a pole, and so are all the other $M_{i}$. So, $(g)_{\infty}=P+M$, and the divisor of zeroes is $(g)_{0}=N+P_{1}$.
    
    So, $g$ is a covering map $X\rightarrow\mathbf{P}^{1}_{k}$ of degree $\pi+1$. Since $P$ is a simple pole, $g$ is not a $p$th-power, and so this covering is separable. Applying \Cref{g1covering}, we get a regular map $X\rightarrow X^{(\pi)}:Q\mapsto H_{Q}$.
    
    
    Let $P_{2}\in X-S\cup \{P,M_{1},...,M_{\pi}\}$, and $a=g(P_{2})$. We have that $P_{2}$ is not a pole of $g$, and $(g-a)\geq P_{2}-P-M$. So, by the proof of \Cref{precompositionlemma}, $g-a$ spans $L_{\md}(P-P_{2}+M)$, and $g-a$ is nonzero on $S$. Define \[g'=(g-a)/(1-a),\] which satisfies $g'\equiv1$ mod $\md$, and $(g')=g^{-1}(a)-(g)_{\infty}=P_{2}+H_{P_{2}}-P-M$.
    
    So, $H_{P_{2}}\sim_{\md}-P_{2}+P+M$. So, $H_{P_{2}}$ is exactly the $N$ in the construction of $\theta$ associated to $D=-P_{2}$ and $M$. So $$\theta(-P_{2})=\varphi(H_{P_{2}})-\varphi(M),$$ and so $$\theta(P_{2})=-\varphi(H_{P_{2}})+\varphi(M),$$ which is regular at $P_{1}$, because up to a translation of $\gj$, it is the composition of $X\rightarrow X^{(\pi)}:Q\mapsto H_{Q}$ and $-\varphi:X^{(\pi)}\dashrightarrow\gj$, both of which are regular where it matters.
    
    So, since $P_{1}$ was arbitrary, $\theta$ is indeed regular on $X-S$.
\end{proof}

Since $\varphi$ and $\theta$ coincide on wherever both are defined, we have that the domain of definition of $\varphi$ on $X^{(\pi)}$ is $(X-S)^{(\pi)}$, and so $\varphi$ and $\theta$ are identical as maps from $X^{(\pi)}$. Following the conventions of [1], we will write $\varphi_{\md}$ for $\theta$ and $\varphi$ in all situations (regardless of the domain).

\section{Universal Property}

\begin{theorem}
    Let $f:X\dashrightarrow G$ be a rational map from $X$ to a commutative group $G$ admitting a modulus $\md$. There exists a unique algebraic homomorphism $F:\gj\rightarrow G$ such that $f=F\circ \varphi_{\md}+f(P)$.
\end{theorem}
\begin{proof}
    After a translation on $f$, we may assume that $f(P)=0$. Let $D$ be a divisor on $X$ prime to $S$. Because $f$ admits $\md$ as a modulus, $f(D)$ only depends on the $\md$-equivalence class of $D$. Since $f(P)=0$, we get a (not yet shown to be algebraic) homomorphism $F:C_{\md}^{0}\rightarrow G$ such that $f=F\circ\varphi_{\md}$ (this holds for all divisors prime to $S$, so in particular this holds for any point in $X-S$).
    
    It remains to show that $F$ is regular. Taking the extension of $\varphi_{\md}$ to $X^{(\pi)}$, we have an inverse $\varphi_{\md}^{-1}$ to $\varphi_{\md}$ (since it is birational) with domain of definition $U$ in $\gj$. In this case, $F=f\circ \varphi_{\md}^{-1}$, which is regular on $U$. By translation, $F$ is regular everywhere.
\end{proof}

\section{Over Arbitrary Fields}

Here we quote the main result in [1], chapter 5, section 4, which allows us to drop the assumption that $X$ has a $k$-rational point outside $S$.

\begin{definition}
    A homogeneous space $H$ for a commutative algebraic group $G$ is a non-empty variety on which $G$ acts algebraically and transitively.
    
    A principal homogeneous space is a homogeneous space where the regular map $G\times H\rightarrow H\times H: (g,h)\mapsto (gh,h)$ is an isomorphism of varieties.
    
    Let $H'$ be a principal homogeneous space for another commutative algebraic group $G'$. A map of principal homogeneous spaces $h:H\rightarrow H'$ is called a morphism if there exists an algebraic homomorphism $h_{0}:G\rightarrow G'$ such that for every $x\in H,g\in G$, $h(x+g)=h(x)+h_{0}(g)$.
\end{definition}

Let $k$ be an arbitrary field, and all varieties below are defined over $k$. Let $H$ be a principal homogeneous space for a group $G$, $X$ (as always) a smooth, projective, geometrically connected curve, and let $f:X\dashrightarrow H$ be a rational map. Let $\md$ be a modulus on $X$.

\begin{definition}
    Let $D$ be a divisor on $X$ prime to $S$ with $\textrm{deg}(D)=0$. We can pair up the degree 1 summands $P_{1},...,P_{n}$ (with possible repeats) and degree -1 summands $Q_{1},...,Q_{n}$ (again, with possible repeats) of $D$, and take $g_{i}$ to be the unique element of $G$ such that $g_{i}Q_{i}=P_{i}$. Define $f(D)=\sum g_{i}$. We say that $f$ admits a modulus $\md$ iff for every $D\sim_{\md}0$, $f(D)=0$.
\end{definition}

\begin{theorem}\label{gjthm}
    There exists an algebraic group $\gj$ (such that when $k$ is perfect, $(\gj)_{\bar{k}}$ (as an algebraic group) is isomorphic to the generalised Jacobian $J_{\bar{\md}}$ associated to $X_{\bar{k}}$ and $\bar{\md}$), and a principal homogeneous space $\gj^{(1)}$ of $\gj$, and a rational map $$\varphi_{\md}:X\dashrightarrow \gj^{(1)}$$ satisfying the following universal property:
    
    Let $f:X\dashrightarrow H$ be a rational map from $X$ to a principal homogeneous space $H$ for a group $G$. If $f$ admits the modulus $\md$, then $f$ can be uniquely factored as $$f=\theta\circ\varphi_{\md}$$ where $\theta:\gj^{(1)}\rightarrow H$ is a morphism of principal homogeneous spaces. 
\end{theorem}

\chapter{Class Field Theory of Curves}

Global class field theory classifies abelian extensions of global fields, which in the case of global function fields, is equivalent to classifying connected coverings of smooth, projective, geometrically connected curves over finite fields.

We will work in slightly more generality, and will classify (possibly non-connected) coverings of smooth, projective, geometrically connected curves over perfect fields.

\section{Abelian Coverings of Varieties}

\begin{definition}

Let $k$ be a perfect field, $V$ be an irreducible $k$-varieties, and $W$ be a (possibly reducible) $k$-variety. We say that $\pi:W\rightarrow V$ is a galois covering if there is a finite group $\g$ acting on $W$ such that $W/\g\cong V$, and the projection map $W\rightarrow W/\g$ is the same as $\pi$, and that the group of automorphisms of $W$ that fix $V$ is $\g$. We say that the covering is abelian if $\g$ is abelian.

Write $n$ for the order of $\g$. We say that a point $P\in V$ is a ramification point (or that it is ramified in $W$) if $\pi^{-1}(P)$ contains repeated points. We say that $P\in V$ is unramified in $W$ if $\pi^{-1}(P)$ contains no repeated point.

\end{definition}

\begin{definition}
Let $H$ (resp. $H'$) be a homogeneous space for a connected commutative algebraic group $G$ (resp. $G'$). An algebraic homomorphism $f:G\rightarrow G'$ is called an isogeny if it is surjective and has finite kernel. A morphism of homogeneous spaces $h:H\rightarrow H'$ is called an isogeny if its corresponding homomorphism $h_{0}:G\rightarrow G'$ is an isogeny of connected commutative algebraic groups. Equivalently, $G\rightarrow G'$ (resp. $H\rightarrow H'$) is an isogeny if the map is surjective and $G,G'$ (resp. $H,H'$) have the same dimension. We say an isogeny is an abelian isogeny if it is also an abelian covering.
\end{definition}

\begin{theorem}\label{coverthm}
    Let $\pi:W\rightarrow V$ be an abelian covering with Galois group $\g$. There exists a rational map $f:V\dashrightarrow G$ and an abelian isogeny $G'\rightarrow G$ such that the following is a pullback diagram:
\[\begin{tikzcd}
	 W \arrow[r,dashed] \arrow[d]{}{\pi} & G'\arrow[d] \\
	 V \arrow[r,dashed]{}{f} & G
\end{tikzcd}\]
\end{theorem}
\begin{proof}
    First we work over the algebraic closure of $k$. Let $\bar{k}(\g)$ denote the group algebra of $\g$ over $\bar{k}$ (considered an affine space), $G'$ be the subvariety of $\bar{k}(\g)$ consisting of invertible elements, and $G$ be the quotient of $G'$ by $\g$.
    
    We need to find a rational map (over $\bar{f}$) $g:W\dashrightarrow G'$ such that \[\begin{tikzcd}
	 W_{\bar{k}} \arrow[r,dashed]{}{g} \arrow[d]{}{\pi} & G'\arrow[d] \\
	 V_{\bar{k}} \arrow[r,dashed]{}{f_{\bar{k}}} & G
    \end{tikzcd}\]
    is a pullback diagram, that is, it is a commutative diagram, and that $g$ commutes with the action of $\g$.
    
    Since we need $g$ to commute with the action of $\g$, $g$ is of the form $x\mapsto (\varphi^{s}(x))$ where $\varphi$ is a rational function on $W_{\bar{k}}$, and where $(\varphi^{s}(x))$ is the vector of the Galois orbit of $\varphi(x)$.
    
    Furthermore, we need $(\varphi^{s}(x))$ to be in $G'$, that is, they are invertible elements of $\bar{k}(\g)$. Since we are looking for a \textit{rational} map $g$, we only need a non-empty open set in $W$ such that $(\varphi^{s}(x))$ is invertible. Since invertibility here is an open condition, we only need to find some $x\in W$, and $\varphi\in \bar{k}(W)$ such that $(\varphi^{s}(x))$ is invertible. We will show that this is equivalent to $\textrm{det}(\varphi^{st}(x))\neq0$.
    
    Write the elements of $\g$ as $\g=\{1=\sigma_{1},\sigma_{2},...,\sigma_{n}\}$. $(\varphi^{s}(x))\in G'$ means that there exists $\alpha_{1},...,\alpha_{n}\in \bar{k}$ such that  $$(\varphi(x)+\varphi(x)\sigma_{2}+...+\varphi(x)\sigma_{n})(\alpha_{1}+\alpha_{2}\sigma_{2}+...+\alpha_{n}\sigma_{n})=1,$$ or equivalently, $(1,0,...,0)$ is in the span of $\{(\varphi^{s}(x)),(\varphi^{s\sigma_{2}}(x)),...,(\varphi^{s\sigma_{n}}(x))\}$, which is satisfied when $\textrm{det}(\varphi_{st}(x))\neq0$.
    
    When $\textrm{det}(\varphi^{st}(x))=0$, there exists $\beta_{1},...,\beta_{n}\in\bar{k}$ such that $$\beta_{1}(\varphi^{s}(x))+\beta_{2}(\varphi^{s\sigma_{2}}(x))+...+\beta_{n}(\varphi^{s\sigma_{n}}(x))=0,$$ or equivalently, $$(\varphi(x)+\varphi(x)\sigma_{2}+...+\varphi(x)\sigma_{n})(\beta_{1}+\beta_{2}\sigma_{2}+...+\beta_{n}\sigma_{n})=0,$$ and $(\varphi^{s}(x))\notin G'$.
    
    Now we show that there exists such $\varphi$. The normal basis theorem (which extends easily to the case of a Galois algebra over a field) for $k(W)|k(V)$ guarantees a function $\varphi'\in k(W)$ whose Galois orbit forms a basis for the Galois algebra $k(W)|k(V)$. After a base change, the Galois orbit of $\varphi_{\bar{k}}$ is a basis for $\bar{k}(W)|\bar{k}(V)$. Write $\varphi=\varphi'_{\bar{k}}$. The theory of discriminants gives that $\textrm{det}(\varphi^{st})\neq0\in\bar{k}(W)$ (this is the content of [8], chapter 3, proposition 9). In particular, there exists some $x\in W$ such that $\textrm{det}(\varphi^{st}(x))\neq0$. This proves the theorem in the case of an algebraically closed base field.
    
    $G'$ can be defined over $k$, as it is just the complement of an affine variety defined by the vanishing of a determinant. Taking the quotient by $\g$, we see that $G$ can also be defined over $k$. $g$ is also defined over $k$, as on each coordinate over $\bar{k}$, it is simply a regular function defined by coefficients in $k$. So, the theorem is true over $k$ as well.
\end{proof}

\begin{proposition}\label{urrgl}
    With the setting of \Cref{coverthm}, if $P\in V$ is unramified in $W$, then the map $f$ can be chosen to be regular at $P$.
\end{proposition}
\begin{proof}
    By the proof of \Cref{coverthm}, we need to be show that it is possible to choose a $\varphi\in k(W)$ whose Galois orbit forms a basis for the extension $k(W)|k(V)$, and that $\varphi$ is regular at all $Q\in \pi^{-1}(P)$.
    
    Let $\sheaf_{P}$ be the local ring of $V$ at $P$, $\m_{P}$ its maximal ideal, and $\sheaf'_{P}$ its integral closure in $k(W)$. Write $k'(P)\defeq\sheaf'_{P}/\m_{P}\sheaf'_{P}$. Because $P$ is unramified, $k'(P)$ is a Galois algebra over $k(P)$ of degree $n$.
    
    Let $\lambda\in k'(P)$ whose Galois orbit forms a basis over $k(P)$ (which exists by the normal basis theorem). Let $\varphi\in\sheaf'_{P}$ be a representative of $\lambda$. $\varphi^{s}$ is a normal basis for $k(W)|k(V)$, and $\varphi$ is invertible in $k'(P)$, thus invertible in $\sheaf'_{P}$, that is, has no pole at any $Q\in\pi^{-1}(P)$.
\end{proof}

\section{Abelian Coverings of Curves}

Let $k$ be a perfect field. Let $X$ be a smooth, projective, geometrically connected curve, and $\pi:X'\rightarrow X$ an abelian covering with Galois group $\g$. Let $S\subset X$ be the finite set of ramification points of $\pi$. By the results of the previous section, we have a pullback diagram: 
\[\begin{tikzcd}
	 X' \arrow[r,dashed]{}{g} \arrow[d]{}{\pi} & G'\arrow[d] \\
	 X \arrow[r,dashed]{}{f} & G
\end{tikzcd}\]
which we can (by \Cref{urrgl}) represent by the following pullback diagram with regular maps:
\[\begin{tikzcd}
	 X'-\pi^{-1}(S) \arrow[r]{}{g} \arrow[d]{}{\pi} & G'\arrow[d] \\
	 X-S \arrow[r]{}{f} & G
\end{tikzcd}\]
where $f$ is regular outside $S$. By Theorem 1, there exists a modulus $\md$ on $X$ with support is $S$. By considering $G,G'$ as principal homogeneous spaces of themselves, and by \Cref{gjthm}, $f$ factors as follows:
\[\begin{tikzcd}
    X-S \arrow[r]{}{\varphi_{\md}} &\gj^{(1)} \arrow[r]{}{\theta} & G
\end{tikzcd}\]
where $\theta$ is a morphism of principal homogeneous spaces. By taking the pullback of $\theta:\gj^{(1)}\rightarrow G$ and $G'\rightarrow G$, we get an abelian isogeny $H\rightarrow \gj^{(1)}$ from a principal homogeneous space $H$ of some group. By a further pullback of $\varphi_{\md}:X-S\rightarrow \gj^{(1)}$ and $H\rightarrow \gj^{(1)}$, we get a commutative diagram of 2 pullback squares:
\[\begin{tikzcd}
    X'-\pi^{-1}(S) \arrow[r] \arrow[d]{}{\pi} & H \arrow[r] \arrow[d] & G'\arrow[d] \\
	X-S \arrow[r]{}{\varphi_{\md}} &\gj^{(1)}\arrow[r] & G.
\end{tikzcd}\]
In particular, we have proved the following:

\begin{theorem}\label{mainthm}

Any abelian covering $\pi:X'\rightarrow X$ is the pullback of an abelian isogeny onto $\gj^{(1)}$ for some $\md$ whose support $S$ is contained in the set of ramification points:
\[\begin{tikzcd}
	 X' \arrow[r,dashed] \arrow[d]{}{\pi} & H\arrow[d] \\
	 X \arrow[r,dashed]{}{\varphi_{\md}} & \gj^{(1)}.
\end{tikzcd}\]

\end{theorem}

\section{Uniqueness}

In this section and the next, we prove some uniqueness statements (\Cref{firstunique} and \Cref{conductor}), which finishes the classification of abelian coverings of a smooth, projective, geometrically connected curve.

It would be really nice if there is a bijection between \{Abelian isogenies onto $\gj^{(1)}|\md$ is a modulus on $X$\} and \{Abelian coverings of $X$\}. Unfortunately this is not true, because if a covering is a pullback of an abelian isogeny onto $\gj^{(1)}$, then the map $\varphi_{\md}$ also admits any moduli $\md'\geq\md$, hence factoring through them, giving an abelian isogeny onto $J_{\md'}$ which pulls back to whatever covering we started with, as depicted below, where both squares and the big rectangle are pullbacks:
\[\begin{tikzcd}
    X' \arrow[r,dashed] \arrow[d]{}{\pi} & H' \arrow[r] \arrow[d] & H\arrow[d] \\
	X \arrow[r,dashed]{}{\varphi_{\md'}} & J_{\md'}^{(1)}\arrow[r] & \gj^{(1)}.
\end{tikzcd}\]

Our first uniqueness statement will be working over a fixed modulus.

\begin{proposition}\label{firstunique}
    Let $\md$ be a modulus on $X$, and $\g$ be a finite abelian group. Write $\textrm{Cov}(X,\g)$ for the set of abelian coverings of $X$ with Galois group $\g$, and $\textrm{Ext}(\gj^{(1)},\g)$ for the set of abelian isogenies onto $\gj^{(1)}$ with Galois group $\g$. The map $\phi^{*}$ that takes an element of $\textrm{Ext}(\gj^{(1)},\g)$ to its pullback along $\varphi_{\md}$ in $\textrm{Cov}(X,\g)$ is an injection.
\end{proposition}
\begin{proof}
    (Sketch) Using the Baer sum operation, the set of abelian isogenies onto $\gj^{(1)}$ with Galois group $\g$ is an abelian group, which we call $\textrm{Ext}(\gj^{(1)},\g)$. The Baer sum operation also defines an abelian group structure on $\textrm{Cov}(X,\g)$, the set of abelian coverings of $X$ with Galois group $\g$. It can be easily checked $\phi^{*}$ is a homomorphism. We need to check that $\phi^{*}$ has trivial kernel.
    
    Let $G\rightarrow \gj^{(1)}$ be an abelian isogeny whose pullback is the trivial covering of X, that is, $X'$ is a coproduct of $n$ copies of $X$. Write $X'=X_{1}\cup...\cup X_{n}$. There are sections $s_{i}:X\rightarrow X'$ that sends $X$ to $X_{i}$, which also make the following diagram commutative:
    \[\begin{tikzcd}
	 X' \arrow[r,dashed] & H\arrow[d] \\
	 X \arrow[r,dashed] \arrow[u]{}{s} & \gj^{(1)}.
    \end{tikzcd}\]
    By composing $s_{i}$ and the top map, we get a map $X\rightarrow H$ that admits $\md$ as a modulus. Factoring this through $\gj^{(1)}$, we get a map $\theta_{i}:\gj^{(1)}\rightarrow H$ such that the composition of
    \[\begin{tikzcd}
    \gj^{(1)} \arrow[r]{}{\theta_{i}} & H \arrow[r] & \gj^{(1)}
    \end{tikzcd}\]
    is the identity. These sections show that $H=\gj^{(1)}\times\g$. Hence $\varphi^{*}$ has trivial kernel, which is what we wanted.
\end{proof}

\begin{corollary}
    The set of unramified coverings of $X$ is in bijection with the set of abelian isogenies onto $J$ (the Jacobian variety, which is $\gj$ when $\md=0$).
\end{corollary}
\begin{proof}
    By \Cref{mainthm}, any unramified covering of $X$ can be obtained as the pullback of an abelian isogeny onto $J$. By \Cref{firstunique}, the coverings obtained this way are unique.
\end{proof}


    
    

\section{Uniqueness Continued: Conductors}

Our next and final goal is to prove that for every abelian covering of $X$, there exists a smallest modulus $\md$ (called the conductor of the covering) such that the covering is a pullback of an abelian isogeny onto $\gj^{(1)}$. For this purpose we need to study how generalised Jacobians relate to each other.

Once again we work over the algebraic closure. Assume $k=\bar{k}$ for now (so that we are guaranteed a point $P\in X-S$ for any finite set $S$). 

\begin{lemma}\label{reallastlemma}
    Let $\md\geq\md'$ be moduli on $X$. There exists a unique surjective morphism $F:\gj^{(1)}\rightarrow J_{\md'}^{(1)}$ such that $\varphi_{\md'}=F\circ\varphi_{\md}$ (whose kernel we call $H_{\md/\md'}$).
\end{lemma}
\begin{proof}
    The existence and uniqueness of $F$ is clear by \Cref{gjthm}. We need to check that $F$ is surjective with connected kernel. We will check that the corresponding algebraic homomorphism $F_{0}$ is surjective, which will suffice.
    
    Let $P\in X$ be a point outside the support of both moduli. Define $s:X^{(\pi')}\rightarrow X^{(\pi)}$ by \[M_{1}+...+M_{\pi'}\mapsto M_{1}+...+M_{\pi'}+(\pi-\pi')P.\] $s$ defines a rational map $s':J_{\md'}\rightarrow\gj$. Since $\gj$ (resp. $J_{\md'}$) is birationally isomorphic to $X^{(\pi)}$ (resp. $X^{(\pi')}$), and $s$ composed with the quotient map $X^{(\pi)}\rightarrow X^{(\pi')}$ is the identity, we have that $F\circ s'=1$.
    
    The section $s'$ shows the surjectivity of $F_{0}$, and thus the surjectivity of $F$.
\end{proof}

Next we study the structures of the groups $H_{\md/\md'}$. By \Cref{gjdvs}, $J_{\md}$ (resp. $J_{\md'}$) is isomorphic to the group of $\md$-isomorphism classes of divisors prime to the support $S$ (resp. $S'$) of $\md$ (resp. $\md'$) of degree $0$. So, if $\md'\neq0$, then $H_{\md/\md'}$ is isomorphic to the group of degree 0 $\md$-equivalence classes of divisors that are $\md'$-equivalent to $0$, that is, the group of $\md$-equivalence classes of divisors that are of the form $(g)$ for some rational function $g$ satisfying $v_{P}(1-g)\geq n'_{P}$ (here we write $\md'=\sum n'_{P}P$), and if $\md'=0$, then $H_{\md/\md'}$ isomorphic to the group of $\md$-equivalence classes of divisors of the form $(g)$ where $g$ is a non-constant rational function.

\begin{lemma}\label{realreallastlemma}
    Let $\md'$ and $\m''$ be 2 moduli such that the abelian covering $\pi:X'\rightarrow X$ is the pullback of an abelian isogeny onto $J_{\md'}$ and an abelian isogeny onto $J_{\md''}$. Then $Y$ is the pullback of an abelian isogeny of $\gj$ where $\md=\textrm{Inf}(\md',\md'')$.
\end{lemma}
\begin{proof}
    Suppose $\md,\md'>0$. Put $\md_{1}=\textrm{Sup}(\md',\md'')$. We have the following sequence of canonical homomorphisms provided by \Cref{reallastlemma}:
    \[\begin{tikzcd}
    & J_{\md'} \arrow[dr]\\
    J_{\md_{1}} \arrow[ur] \arrow[dr] && \gj \arrow[r] & J.\\
    & J_{\md''} \arrow[ur]
    \end{tikzcd}\]
    Write $H',H'',H,H_{1}$ for the kernels of the canonical homomorphisms from $J_{\md_{1}}$ to $J_{\md'},J_{\md''},J_{\md}, J$. Clearly $H'\subset H,H''\subset H,H\subset H_{1}$. By the structures of such kernels as discussed before, we have that $H',H''$ together generate $H$, and that the canonical map $H'\times H''\rightarrow H$ is an isomorphism of varieties (if $\textrm{deg}(\md)>0$), or $H$ is isomorphic to a quotient $H'\times H''/\mathbf{G}_{m}$ (if $\md=0$).
    
    Let $J'\rightarrow J_{\md'}$ and $J''\rightarrow J_{\md''}$ be abelian isogenies whose pullbacks are $\pi:X\rightarrow X$. Write $J_{1}',J_{1}''$ for the pullbacks of $J'\rightarrow J_{\md'}$ and $J''\rightarrow J_{\md''}$ along the canonical homomorphisms $J_{\md_{1}}\rightarrow J_{\md'}$ and $J_{\md_{1}}\rightarrow J_{\md''}$, as depicted below (for the case of $\md'$):
    \[\begin{tikzcd}
    X' \arrow[r,dashed] \arrow[d]{}{\pi} & J_{1}' \arrow[r] \arrow[d] & J'\arrow[d] \\
	X \arrow[r,dashed]{}{\varphi_{\md_{1}}} & J_{\md_{1}} \arrow[r] & J_{\md'}.
    \end{tikzcd}\]
    By \Cref{firstunique}, we have that $J_{1}'\cong J_{1}''$, since they both have abelian isogenies onto $J_{\md_{1}}$ pulling back to $\pi:X'\rightarrow X$. We thus denote both of them by $J_{1}$.
    
    By mapping $H',H''$ to $J_{\md'},J_{\md''}$, we see that the abelian isogeny $J_{1}\rightarrow J_{\md_{1}}$ is trivial on $H'$ and $H''$. Let $s':H'\rightarrow J_{1}$ and $s'':H''\rightarrow J_{1}$ be sections that witness their triviality. Write $s:H'\times H'' \rightarrow J_{1}$ be the sum of $s'$ and $s''$. Let $Q$ be the kernel of $H'\times H''\rightarrow H$ (which is either 0 or $\mathbf{G}_{m}$). The commutative diagram
    \[\begin{tikzcd}
    && J_{1} \arrow[d]\\
    H'\times H'' \arrow[urr]{}{s} \arrow[r] & H \arrow[r] & J_{\md_{1}}
    \end{tikzcd}\]
    shows that $s$ maps $Q$ to the kernel $\g$ of the abelian isogeny $J_{1}\rightarrow J_{\md_{1}}$. Since $Q$ is connected (as it is either 0 or $\mathbf{G}_{m}$), $s(Q)=0$. So $s$ factors as the following commutative diagram shows:
    \[\begin{tikzcd}
    && J_{1} \arrow[d]\\
    H'\times H'' \arrow[urr]{}{s} \arrow[r] & H \arrow[ur]{}{t} \arrow[r] & J_{\md_{1}},
    \end{tikzcd}\]
    So the isogeny $J_{1}\rightarrow J_{\md_{1}}$ is also trivial.
    
    Then an Ext long exact sequence of commutative algebraic groups shows that the abelian isogeny $J_{1}\rightarrow J_{\md_{1}}$ is the pullback of an abelian isogeny $G\rightarrow \gj$.
\end{proof}

\begin{proposition}\label{conductor}
    Let $k$ be a perfect field. Let $\pi:X'\rightarrow X$ be a smooth, projective, geometrically connected curve. There exists a modulus $\md$ on $X$ for which there is an abelian isogeny onto $\gj^{(1)}$ that pulls back to $\pi$, and that for any other $\md'$ such that there is an abelian isogeny onto $J_{\md'}^{(1)}$ pulling back to $\pi$, $\md\leq\md'$.
    
    Furthermore, the support $S$ of $\md$ is the set of ramification points of $\pi$ in $X$.
\end{proposition}
(Such $\md$ is called the conductor for the covering $\pi$)
\begin{proof}
    In the algebraically closed case, there is no difference between working with $(\gj)_{\bar{k}}$ and $(\gj^{(1)})_{\bar{k}}$. Repeatedly applying \Cref{realreallastlemma}, we obtain the smallest modulus $\md_{0}$. By \Cref{urrgl}, $S_{0}$ is contained in the set of ramification point of $\pi_{\bar{k}}$ in $X_{\bar{k}}$. On the other hand, the pullback of an abelian isogeny along a map regular at a point $P$ gives a covering that is unramified at $P$. So $S_{0}$ is equal to the set of ramification points of $\pi_{\bar{k}}$.
    
    Let $\bar{\md}$ be the modulus on $X_{\bar{k}}$ that is the smallest modulus $\geq\md_{0}$ such that $\bar{\md}$ is invariant under the Galois group $\mathrm{Gal}(\bar{k}|k)$. Then $\bar{\md}$ corresponds to a modulus $\md$ on $X$, which is exactly the modulus we want.
\end{proof}


\include{chapter1}
\include{chapter2}
\include{conclusions}

\appendix
\include{appendix1}
\include{appendix2}

\addcontentsline{toc}{chapter}{Bibliography}
\bibliography{refs}        
\bibliographystyle{plain}  

\end{document}